\documentclass{amsart}
\usepackage{amssymb}
\usepackage[all]{xy}

\newcommand{\Z}{\mathbf{Z}}
\newcommand{\Ps}{\mathbf{P}}
\newcommand{\C}{\mathbf{C}}
\newcommand{\Q}{\mathbf{Q}}
\newcommand{\cT}{\mathcal{T}}

\newtheorem{theorem}{Theorem}[section]
\newtheorem{lemma}[theorem]{Lemma}
\newtheorem{proposition}[theorem]{Proposition}

\theoremstyle{definition}
\newtheorem{definition}[theorem]{Definition}
\newtheorem{notation}[theorem]{Notation}
\newtheorem{assumption}[theorem]{Assumption}

\newtheorem{example}[theorem]{Example}
\theoremstyle{remark} 
\newtheorem{remark}[theorem]{Remark}

\DeclareMathOperator{\coker}{coker}
\DeclareMathOperator{\prim}{prim}
\DeclareMathOperator{\sing}{sing}
\DeclareMathOperator{\Proj}{Proj}
\DeclareMathOperator{\Spec}{Spec}
\DeclareMathOperator{\sat}{sat}
\DeclareMathOperator{\Gr}{Gr}
\DeclareMathOperator{\ima}{im}

\begin{document}

\title[Alexander polynomials and deformations]{Deformations of hypersurfaces with non-constant Alexander polynomial}
\author{Remke Kloosterman}

\email{klooster@math.unipd.it}
\address{Universit\`a degli Studi di Padova,
Dipartimento di Matematica ``Tullio Levi-Civita",
Via Trieste 63,
35121 Padova, Italy}
\thanks{The author would like to thank Orsola Tommasi for comments on a previous version of this paper. The author would like to thank the referees for various suggestions to improve the exposition.}

\begin{abstract} 
Let $X \subset \Ps^n$ be an irreducible hypersurface of degree $d\geq 3$ with only isolated semi-weighted homogeneous singularities,  such that $\exp(\frac{2\pi i}{k})$ is a zero of its Alexander polynomial. Then we show that the equianalytic deformation space of $X$ is not $T$-smooth  except for a finite list of triples $(n,d,k)$.

This result captures the very classical examples  by B. Segre of families  of degree $6m$ plane curves with $6m^2$, $7m^2$, $8m^2$ and $9m^2$ cusps, where $m\geq 3$.

Moreover, we argue that many of the hypersurfaces with non-trivial Alexander polynomial are limits of constructions of hypersurfaces with not $T$-smooth deformation spaces. In many instances this description can be used to find  candidates for  Alexander-equivalent Zariski pairs.
\end{abstract}

\maketitle

\section{Introduction}\label{secIntro}
Let $X\subset \Ps^n$ be a hypersurface with isolated singularities and let $\Delta_X\in \Z[t]$ be its Alexander polynomial (cf. Defintion~\ref{defAlexPol}).

The first example in  the literature of a hypersurface with non-constant Alexander polynomial is a plane sextic with 6 cusps on a conic, due to Zariski \cite{ZarSyz}. One easily checks  that there exist  sextic curves with six cusps such that these six cusps are not on a conic. Such a sextic has a constant Alexander polynomial. Hence we obtain a Zariski pair, a pair of singular hypersurfaces $X_1,X_2\subset \Ps^n$  with the same combinatorial data, but such that there is no homeomorphism between the pairs $(\Ps^n,X_1)$ and $(\Ps^n,X_2)$.  One easily checks that in this case both curves have a $T$-smooth equianalytic deformation space, see \cite[Section VIII.5]{ZarAlgSur}.

B. Segre  noted that one can easily generalise this example to higher degree and considered two families of degree $6m$ curves with $6m^2$ cusps. For the first construction, pick two sufficiently general homogenous polynomials  $f,g\in \C[x_0,x_1,x_2]$ of degree $2m$ and $3m$ respectively. Then the curve of degree $6m$ given by
\[ f^{3}+g^{2}=0\]
has $6m^2$ cusps. Within the space of degree $6m$ curves the expected codimension of the locus of curves with $6m^2$ cusps is $12m^2$,  however, one easily shows that this family has codimension $12m^2-\frac{1}{2}(m-1)(m-2)$. I.e., for $m>2$ this family has a larger equisingular deformation space than expected.  Similar examples are due to Segre for degree $6m$ curves with either $7m^2$, $8m^2$ ot $9m^2$ cusps. (For more on this see \cite[Section VIII.5]{ZarAlgSur}.) These examples  have non-constant Alexander polynomial and their equisingular deformation spaces  have larger dimension than expected. We will come back to these examples in Example~\ref{exaConstr}.

Take now a sextic curve $C$ with six cusps not on a conic and pull this curve back under  a general self-map of $\Ps^2$ of degree $m$.
Then the pullback of $C$ has $6m^2$ cusps. One easily checks that  the saturation of the Jacobian ideal $J^{\sat}$ is $6m$-regular, and therefore the equianalytic deformation space is $T$-smooth. For ordinary cusps the equianalytic and equisingular deformation space coincide, hence also the latter space is $T$-smooth.
Hence for every $m>2$ we have two ways to show that the space parametrizing  curves of degree $6m$ with $6m^2$ cusps is reducible, i.e., we can detect this both by the Alexander polynomial and by the dimension of the equisingular deformation space.

There have been more ``recent" attempts to explain this excess dimension  of the deformation space (e.g., \cite{Tan}). We will give an explanation different from those we  were able to locate in the literature. 
Our main result states that Segre's construction is part of a rather frequently occurring phenomenon:
\begin{theorem}
Let $(n,d,k)$ be integers such that $d\geq 3,n\geq 2$ and $k\geq 1$. 
Suppose $X\subset \Ps^n$ is an irreducible hypersurface of degree $d$ with isolated semi-weighted homogeneous singularities (cf. Definition~\ref{defSing}) such that $\exp(2\pi i/k)$ is a zero of its Alexander polynomial of $X$.

Moreover, assume that we are \emph{not} in one of the following cases
\begin{enumerate}
\item $n=2$, $d\in \{6,12\}$ and $k=6$.
\item $n\in \{3,4,6\}$, $d=3$ and $k=1$; 
\item $n\in \{3,4,5\}$, $d=3$ and $k=3$;
\item $n=3$, $d\in \{4,6\}$ and $k=2$;
\item $n=3$, $d=k=4$;
\item $n=4$, $d=4$ and $k=1$;
\end{enumerate}
Then the equianalytic deformation space of $X$ is not $T$-smooth.
\end{theorem}

If $d=1$ then $X$ is a hyperplane and therefore $X$ is smooth. If $d=2$ and $X$ has isolated singularities then $X$ is a quadric of rank one less than the maximal rank. In this case the deformation theory is very simple. For this reason we considered only the case $d\geq 3$.

Unfortunately, our methods only show that the tangent space has dimension larger than expected and our method can only be applied to the equianalytic deformation space. 
In particular, this result  does not show that the equisingular deformation space has dimension larger than expected, but it is strong evidence of it. More precisely,  let $X$ be a hypersurface with nonconstant  Alexander polynomial and let $X'$ be an equisingular deformation of $X$.  Then the Alexander polynomials of $X$ and $X'$ are the same. Hence the equianalytic deformation space of every equisingular deformation of $X$ is nowhere $T$-smooth and therefore each of these spaces is non-reduced or has dimension larger than expected. In case there exists a further hypersurface $X''$ with the same combinatorial data as $X$, but with a $T$-smooth equianalytic deformation space then the space of hypersurfaces with this  combinatorial data has at least two irreducible components.

In the final section we will provide constructions of hypersurfaces with deformation spaces with dimension larger than expected. These constructions depend on choices of several parameters, each of which are integers, subject to several inequalities. For most choices of parameters the resulting hypersurface has constant Alexander polynomial, but for a few choices of these parameters the corresponding hypersurface has nonconstant Alexander polynomial. Moreover, in the latter case at least one of the before-mentioned inequalities turns out to be an equality, i.e., the examples with nonconstant Alexander polynomials can be considered to be  boundary cases or limit cases.
 This strongly suggests that the Alexander polynomial is not an optimal invariant to detect examples of reducible spaces parametrizing  singular hypersurfaces with fixed combinatorial data. The dimension of the equisingular  deformation space seems a better invariant.  On the other hand, certain geometric phenomena (e.g.,  quasi-torus structures, the relation with Mordell-Weil ranks of isotrivial fibrations) can only occur for hypersurfaces with non-constant Alexander polynomials, see \cite{CogLib, TorDec}.

The assumption that all singularities are semi-weighted homogeneous is needed in our proof, as we heavily use the fact that for each singular point the pole-order filtration and the Hodge filtration on the cohomology of the Milnor fiber coincide. If one could control the difference between these two filtrations for other types of singularities then one might be able to extend our approach to larger classes of singularities.

The proof of the main result consists of two parts. In the first part we reconsider Dimca's approach \cite{DimBet,Dim} to calculate the Alexander polynomial of a hypersurface with isolated semi-weighted homogenous singularities. This is done in Section~\ref{SecCalCoh}.
The upshot of this method is that if $\exp(2\pi i/k)$ is a zero of its Alexander polynomial and $J$ is the Jacobian ideal of $X$, then its saturation $J^{\sat} $ has defect in every degree $\leq \alpha(n,d,k) d-n-1$. The rational number $\alpha(n,d,k)$ will be introduced in Section~\ref{SecComb}, but for the rest of this Introduction it suffices to know that $\alpha(n,d,k)$ lies in the interval $[\frac{n}{2},\frac{n+1}{2}]$.

In Section~\ref{SecComb} we determine all values $(n,d,k)$ for which  $\alpha(n,d,k) d-n+1<d$. Except for the case $n=2,k=1$ (reducible plane curves) and $d=2$ (quadric cones) there are only finitely many triples $(n,d,k)$ for which this inequality holds. This part is a purely combinatorial exercise.
It then easily follows that the tangent space of the equisingular deformation space is larger than expected, except for these exceptional values of $(n,k,d)$ .

In Section~\ref{SecExa} we discuss some examples and explain some of the remarks made above in more detail.

\section{Calculation of $H^n(X)$}\label{SecCalCoh}
In this section we discuss Dimca's method to calculate the mixed Hodge structure  on the cohomology of hypersurfaces with isolated semi-weighted homogeneous singularities, see \cite{DimBet} and \cite[Section 6.3]{Dim}. At certain instances we differ slightly from Dimca's approach and for that reason we recall large part of the construction. 

\begin{notation}
Let $n\geq 2$ and let $R=\C[x_0,\dots,x_n]$ be the polynomial ring in $n+1$ variables, with its natural grading. Let $d\geq 1$ be an  integer.  For $f\in R_d$ let $X=V(f)\subset \Ps^n$ be the associated hypersurface.
Let $U=\Ps^n\setminus X$ and let $X^*=X\setminus S$, where $S=X_{\sing}$. 
\end{notation}

\begin{assumption}
For the rest of this section we assume that $f$ is chosen such that $X$ has isolated singularities, i.e., that $S$ is finite.
\end{assumption}

\begin{remark}
Using the Lefschetz hyperplane theorem and a result by Kato-Matsu\-moto \cite[Theorem 5.2.6 and 5.2.11]{Dim} we can determine $H^j(X,\C)$ for all $j\neq n-1,n$. In the rest of this section our focus will be on the case $j=n$, since this group is used to determine the Alexander polynomial of $X$, as we will see in the next section. More precisely, we  aim to give  an upper bound for the dimension of the graded pieces $\Gr_F^s(H^n(X))$ of the Hodge filtration $F$.

If $n=2$ then $H^2(X)$ is of pure type $(1,1)$ and the dimension equals the number of irreducible components of $X$. Hence we may assume for the moment that $n\geq 3$.
\end{remark}

\begin{notation}
Denote with $H^j(X)_{\prim}$ and $H^j(X^*)_{\prim}$ the primitive cohomology as defined in \cite[Section 2]{DimBet}. 
\end{notation}
 By construction, the groups 
  $H^j(X)_{\prim}$ and $H^j(X^*)_{\prim}$ are sub-Hodge structures of $H^j(X)$ and $H^j(X^*)$ respectively. The following result can be found at \cite[page 291]{DimBet}:
  
\begin{proposition}\label{prpPrim} Let $\iota :X^*\to X$ be the inclusion map. Then for all $j\neq 2n-2$ we have that the kernel (respectively the cokernel) of
$\iota^* : H^j(X)\to H^j(X^*)$
equals the kernel (respectively the cokernel) of
$ \iota^*: H^j(X)_{\prim} \to H^j(X^*)_{\prim}.$

Moreover, $H^n(X^*)_{\prim}=0$. 
\end{proposition}

\begin{notation}
For a proper subset $W$ of $X$ denote with $H^n_W(X)$ the cohomology of $X$ with support in $W$.  If $W=\{p_1,\dots,p_l\}$ is a finite set then using excision it follows easily that $H^n_W(X)=\oplus_{i=1}^l H^n_{p_i}(X)$.

Let $\vartheta: H^n(U)(1)\to H^n_S(X)$ be the composition of the following maps
\[ H^n(U)(1)\stackrel{\cong}{\longrightarrow} H^{n-1}(X^*)_{\prim} \to H^{n-1}(X^*)\stackrel{\delta^{n-1}}{\longrightarrow} H^n(X,X^*)\stackrel{\cong}{\longrightarrow} H^n_S(X),\]
where the first map is the Poincar\'e residue map, the second map is the natural inclusion, the third map is the connecting homomorphism of the sequence of the pair $(X,X^*)$ 
\[ H^{n-1}(X,X^*)\to H^{n-1}(X)\to H^{n-1}(X^*)\stackrel{\delta^{n-1}}{\longrightarrow} H^n(X,X^*)\to H^n(X) \to H^{n}(X^*) .\]
and the fourth map is the natural isomorphism $H^n(X^*)_{\prim}\to H^n_{S}(X)$.
\end{notation}

\begin{lemma}\label{lemCoker}\cite[Equation (2.3)]{DimBet} Suppose $n\geq 3$.
The map  $\vartheta: H^n(U)(1)\to H^n_S(X)$ is a natural morphism of MHS. Moreover,  $\coker(\vartheta)\cong H^n(X)_{\prim}$ as MHS.
\end{lemma}
\begin{proof}
The Poincar\'e residue map $H^n(U)(1)\stackrel{\cong}{\longrightarrow} H^n(X^*)_{\prim}$ is a morphism of MHS \cite[Lemma 2.2]{DimBet}. The inclusion $H^n(X)_{\prim} \to H^n(X)$ is a morphism of MHS .The exact sequence of the pair $(X,X^*)$ is an exact sequence of MHS, see \cite[page 291]{DimBet} or \cite{SteQua}. Therefore also $\delta^{n-1}$ is a morphism of MHS. Hence $\vartheta$ is a composition of morphisms of MHS and is itself a morphism of MHS.

We have the following identifications
\begin{eqnarray*}
H^n(X)_{\prim} &=& \ker(H^n(X)_{\prim}\to H^n(X^*)_{\prim})\\
&=&\ker(H^n(X)\to H^n(X^*)) \\
&=& \ima( H^n_S(X) \to H^n(X))\\
&\cong& H^n_S(X)/\ima(\delta^{n-1})\\
&=&H^n_S(X)/\ima (\delta^{n-1}_{\prim})\\
&=& H^n_S(X)/\ima(\vartheta)=\coker(\vartheta).\end{eqnarray*}
The first equality holds since $H^n(X^*)_{\prim}=0$, the second follows from  Proposition~\ref{prpPrim}, the  third and fourth follow from the long exact sequence of the pair $(X,X^*)$. The fifth equality follows from  Proposition~\ref{prpPrim} and the sixth follows from the fact that the Poincar\'e residue map is an isomorphism (see \cite[Lemma 2.2]{DimBet}).
\end{proof}

In order to determine the cokernel of $\vartheta$ we start by identifying generators for $H^n(U)(1)$. This is relatively straightforward since $U$ is affine and therefore its cohomology is the cohomology of its algebraic de Rham complex. 
Let us define the following $n$-form on $\C^{n+1}\setminus \{(0,0,\dots,0)\}$
\[ \Omega=\left(\prod_{j=0}^n x_j\right) \sum_{j=0}^n (-1)^j \frac{dx_0}{x_0} \wedge \dots \wedge \widehat{\frac{dx_j}{x_j}}\wedge \dots \wedge \frac{dx_n}{x_n}.\]
This form can be used to define $n$-forms on $U$.
Let us consider  $H^n(U)$. Using that $U$ is affine it is easy to show that $H^n(U)$ is spanned by classes $\frac{g}{f^s} \Omega,$
with $g\in R_{sd-n-1}$, $s\in \Z_{>0}$, see \cite[Equation 1.3]{DimBet}. We can filter $H^n(U)$ by the order of the pole, by setting $P^sH^n(U)$ to be the subspace of classes which can be represented by elements of the form
\[  \frac{g}{f^{n-s}} \Omega.\]
 Let $F^\bullet$ be the Hodge filtration on $H^n(U)$.  Deligne and Dimca \cite{DelDim} showed that $F^\bullet \subset P^{\bullet}$. 

On the local side we can proceed similarly. Let $p\in S$. Let $V_p\subset \C^n$ be a neighbourhood of $p$. Assume that we choose local coordinates $z_1,\dots, z_n$ such that $p=(0,\dots,0)$. Let $f_p=0$ be a local equation for $X$ in a neighbourhood of $p$. Let $\Omega_p=dz_1\wedge\dots \wedge dz_n$. Let $U_p=V_p\setminus Z(f_p)$. Then the $n$-forms on $U_p$ can be written as
\[ \frac{g}{f_p^{n-s}} \Omega_p.\]
We can define analogously a pole order filtration on $H^n(U_p)$, see \cite[page 288]{DimBet}. Again  this is a decreasing filtration satisfying $F^\bullet \subset P^\bullet$ \cite[Proposition 6.1.39]{Dim} and therefore we can always find a representative for a given cohomology class such that $0\leq s \leq n$. There is a stronger result in the case of semi-weighted homogeneous singularities. First we recall the definition of  semi-weighted homogeneous singularities:

\begin{definition}\label{defSing} Let $g:(\C^n,0)\to(\C,0)$ be an analytic function germ. Let $(Y,0)=(g^{-1}(0),0)$ be the associated hypersurface singularity. We say that $(Y,0)$ is a \emph{weighted homogeneous singularity} if there exists a weighted homogeneous polynomial $h\in \C[z_1,\dots,z_n]$ and an analytic isomorphism $\varphi :(\C^n,0)\to (\C^n,0)$ such that $(\varphi(h^{-1}(0)),0)=(X,0)$.

We say that $(Y,0)$ is a \emph{semi-weighted homogeneous singularity} if  there exist a  polynomial $h\in \C[z_1,\dots,z_n]$, integers $w_1,\dots,w_n$ and  an analytic isomorphism $\varphi :(\C^n,0)\to (\C^n,0)$ such that $(\varphi(h^{-1}(0)),0)=(Y,0)$ and such that we can write $h=h_0+h_1$ with
\begin{enumerate}
\item $h_0$ has an isolated singularity at the origin;
\item $h_0$ is a weighted homogeneous polynomial with respect to $w_1,\dots,w_n$;
\item each monomial in $h_1$ has weighted degree strictly larger than the weighted degree of $h_0$.
\end{enumerate}

Let $X\subset \Ps^n$ be a hypersurface. Then we say that a point $p\in X$ is a \emph{weighted homogeneous singularity}, respectively a \emph{semi-weighted homogeneous singularity} if there exists an analytic neighborhood $V$ of $p$ in $\Ps^n$ such that $(V\cap X,p)$ is a weighted homogeneous singularity, respectively a semi-weighted homogeneous singularity.
\end{definition}

Suppose now that $V_p$ is chosen sufficiently small such that $Z(f_p)$ is contractible. Then the local Poincar\'e residue map and the long exact sequence for the pair $(Z(f_p),Z(f_p)\setminus\{p\})$ yield isomorphisms of MHS
\[ H^n(U_p)(1)\to H^{n-1}((Z(f_p)\setminus \{p\} )\cap V_p)\to H^n_p(X).\]
 We can use the above isomorphism to define the $P$-filtration on $H^n_p(X)$, see \cite[page 288]{DimBet}.
If $f_p$ is semi-weighted homogeneous then the filtrations $F^\bullet$ and $P^\bullet$ on $H^n(U_p)$ coincide by \cite[page 289]{DimBet}.

\begin{lemma}\label{lemCokerF} Suppose $n\geq 3$. Suppose all singularities of $X$ are semi-weighted homogeneous. We have that $\Gr_F^s H^n(X)_{\prim}$ is isomorphic with the cokernel of
\[ \overline{F^s\vartheta}: F^s H^n(U)(1)\to \oplus_{p\in S} \Gr_P^s H^n_p(X).\]
\end{lemma}
\begin{proof}
The morphism $\vartheta: H^n(U)(1)\to \oplus_{p\in S} H^n_p(X)$ is strict for the Hodge filtration by Lemma~\ref{lemCoker} and \cite[Remark C16]{Dim}. Hence $\Gr_F^s$ of the cokernel equals the cokernel
of
\[ F^s H^n(U)(1)\to \oplus_{p\in S} \Gr_F^s H^n_p(X).\]
Since all singularities are semi-weighted homogeneous we obtain that $F^\bullet$ and $P^\bullet$ coincide on $H^n_p(X)$ by \cite[page 289]{DimBet}.
\end{proof}

We introduced a $P$-filtration on $H^n_p(X)$. The direct sum of these filtrations yields a $P$-filtration on $H^n_S(X)$.
We also introduced a $P$-filtration on $H^n(U)(1)$. The following lemma shows that the morphism of MHS $\vartheta: H^n(U)(1)\to H^n_S(X)$ respects these $P$-filtrations. However,  $\vartheta$ is strict for the Hodge-filtration, but does not need to be strict for the $P$-filtration.

\begin{lemma}\label{lemPfilt} Suppose $n\geq 3$. The morphism $\vartheta: H^n(U)(1)\to H^n_S(X)$
respects the $P$-filtrations on $H^n(U)$ and $H^n_S(X)$.
\end{lemma}
\begin{proof}
Let $p\in S$.
Without loss of generality we may assume  $p=(1:0:0:\dots:0)$ and that we have local coordinates $z_j=x_j/x_0$, for $j=1,\dots n$.

Consider now composition of $\vartheta$ with the natural projection map $H^n_S(X)\to H^n_p(X)$:
\[ H^n(U)(1)\to H^n_p(X).\]
We aim to make this map explicit. Consider  the affine chart $x_0\neq 0$.
Let $\omega\in P^sH^n(U)(1)$.
Pick some representative $\frac{g}{f^{n-s}}\Omega$ for $\omega$.
Let $f_p$ be a local equation for $(X,p)$. Then  locally we can write this form as 
\[ \frac{g_p}{f_p^{n-s}} dz_1\wedge dz_2\wedge \dots dz_n\]
Hence $P^s (H^n(U)(1))\subset P^s H^n_p(X)$.
\end{proof}

\begin{proposition}\label{prpCoKern} Suppose $n\geq 3$. We have that $\Gr_F^s H^n(X)_{\prim}$ is isomorphic to the cokernel of
\[  \Gr_P^s H^n(U)(1) \to \oplus_{p\in S} \Gr_P^s H^n_p(X).\]
\end{proposition}

\begin{proof} 
Let $\tau: H^n(U) \to \oplus_{p\in S} H^n_p(X)$.
From Lemma~\ref{lemCokerF} it follows that $\Gr_F^s H^n(X)_{\prim}$ equals the cokernel of 
\[  \overline{ \tau} : F^s H^n(U)(1) \to \oplus_{p\in S} \Gr_P^s H^n_p(X).\]
From Lemma~\ref{lemPfilt} it follows that this map can be extended to $P^sH^n(U)(1)\supset F^s H^n(U)(1)$,  i.e., the map 
\[   \overline{ \tau} : P^s H^n(U)(1)\to \oplus_{p\in S} \Gr_P^s H^n_p(X)\]
is well-defined. It remains to show that the image of $P^s$ is contained in the image of $F^s$.

From Lemma~\ref{lemPfilt} it follows that
\[ \tau(P^s H^n(U)(1))\subset P^s H^n_p(X)=F^s H^n_p(X).\]
Hence 
\[ \tau(P^s H^n(U)(1)) \subset F^s \tau(H^n(U)).\]
Since $\tau$ is strict for $F$ \cite[Remark C16]{Dim}, we find
\[ F^s\tau(H^n(U)(1))=\tau(F^s H^n(U)(1))\]
and we are done.
\end{proof}

\begin{notation}
Let $f\in R$ then $J(f)\subset R$ is the ideal generated by the partials $\partial{f}/\partial{x_i}$ for $i=0,\dots, n$ and $J^{\sat}(f)$ its saturation with respect to the irrelevant ideal $(x_0,\dots,x_n)$.
If no confusion arises then we will write $J$ and $J^{\sat}$ for $J(f)$ and $J^{\sat}(f)$ respectively.
\end{notation}

\begin{lemma} Suppose $n\geq 3$ and $s<n-1$. There is a natural surjective map
\[ (R/J)_{(n-s)d-n-1} \to \Gr_P^s  H^n(U).\]
\end{lemma}
\begin{proof}
By the definition of the $P^\bullet$-filtration, there is
a surjective map
\[ R_{(n-s)d-n-1} \to P^s H^n(U)\]
for any $s\in \{0,\dots,n-1\}$, sending $g$ to $\frac{g}{f^{n-s}} \Omega$.  
Consider for $i\in \{0,\dots,n\}$ the $n-1$-form
\[ \frac{g}{f^{n-s-1}} \sum_{j\neq i}\epsilon x_i dx_0\wedge \dots \wedge \widehat{dx_j}\wedge \dots  \wedge \widehat{dx_i}\wedge\dots \wedge dx_n\]
with $\epsilon=(-1)^{i+j+1}$ for $i<j$ and $(-1)^{i+j}$ for $i>j$.
Differentiating this form shows that 
\[ \frac{fg_{x_i}-(n-1-s) g f_{x_i}}{f^{n-s}} \Omega\]
is zero in cohomology, in particular,  the image of $J_{(n-s)d-n-1}$ in $P^s H^n(U)(1)$ is contained in $P^{s-1}$. Therefore there is a surjection
\[ (R/J)_{(n-s)d-n-1} \to \Gr_P^s  H^n(U).\]
\end{proof}

For $p\in S$ let $\cT_p$ be the Tjurina algebra of $X$ at $p$. Let $f_p$ and $U_p$ as above.
\begin{lemma} Suppose $n\geq 3$. There is a natural surjective map
$\cT_p \to \Gr_P^s H^n_p(X)$.
\end{lemma}
\begin{proof}
Consider now the map $\C\{z_1,\dots,z_n\}\to P^s H^n(U_p)$ sending $g$ to 
\[ \frac{g}{f_p^{n-s}} \Omega_p.\]
Obviously, the ideal generated by $f_p$ lands in $P^{s+1}$. Differentiating  for $j\in\{1,\dots,n\}$ the $(n-1)$-form
\[ \frac{1}{f_p^{n-s-1}} dz_1\wedge dz_2\wedge \dots\wedge \widehat{dz_j}\wedge \dots \wedge dz_n.\]
yields that  $J(f_p)$ lands in $P^{s+1}$. Hence we obtain a well-defined map from the Tjurina algebra of $X$ at $p$ to $\Gr_P^s H^n(U_p)$.\end{proof}

\begin{remark}\label{rmkSplit}
 If $f_p$ is weighted homogeneous with weights $w_i$ and degree $d_p$ then  the map 
 \[ \cT_p \to \Gr_P^s H^n_p(X)\]
 has a natural section. This allows
  us to identify 
\[\Gr_P^s H^n(U_p)(1) \mbox{ with } (\cT_p)_{(n-s)d_p-\sum w_i}\]
However, this latter fact is only used in the examples. For more details, see \cite[Example 3.6]{DimMilnorWS}.
\end{remark}
\begin{proposition}\label{prpOne} Suppose $n\geq 3$. Let $s\in \{0,\dots, n-2\}$. The dimension of $\Gr_F^s H^n(X)$ is at most the defect of $J^{\sat}$ in degree $(n-s)d-n-1$.
\end{proposition}

\begin{proof}
Recall that there is a natural map from the global coordinate ring to each of the local rings. The Jacobian ideal of $f$ is generated  
 by the $n+1$ partials of $f$, whereas ideal defining the Tjurina algebra is generated by  $n$ partials of $f_p$ and $f_p$ itself.
Assume for the moment that $p=(1:0:\dots:0)$.
Let $f_i$ be the partial of $f$ with respect to $x_i$. Then the global Jacobian ideal is generated by $(f_0,\dots,f_n)$, whereas the ideal generating the Tjurina algebra is generated by $(f,f_1,\dots,f_n)$ where we substitute $x_0=1$ and write in local coordinates. 

 Writing the Euler relation $d f(x_0,\dots,x_n)=\sum x_if_i$ in local coordinates we obtain that $df(1,z_1,\dots,z_n)\equiv   f_0 \bmod (f_1,\dots, f_n)$ in $\C\{z_1,\dots,z_n\}$. Therefore there is a well-defined natural map $(R/J)_{(n-s)d-n-1}\to \cT_p$.

Consider now 
\[ H^n(U)(1) \to H^n(U_p)\to H^{n-1}(X^* \cap U_p)\to H^n_p(X).\]
This map factors  through the natural restriction map of forms $H^n(U)\to H^n(U_p)$, which respects the $P$-filtration, hence we have a commutative diagram
\[\xymatrix{
 (R/J)_{(n-s)d-n-1}  \ar[r]^{\tau}\ar[d]& \oplus_{p\in S} \cT_p \ar[d]\\
  \Gr_P^s H^n(U)(1) \ar[r]^{\overline{\tau}}& \oplus_{p\in S}\Gr_P^s H^n_p(X).}\]
Since both vertical maps are surjective, the cokernel of the bottom row is a quotient of the cokernel of the top row. Moreover, in a neighbourhood of $p$ one can identify $\Proj(R/J)$ with $\Spec(\cT_p)$. In particular the scheme  $V(J)$ is just $V(\oplus \cT_p)$ and therefore the kernel of the map in the top row equals $J^{\sat}/J$. Let $\xi$ be the length of $V(J)=V(J^{\sat})$, the total Tjurina number of the singularities of $X$. Then the cokernel of the top row has dimension
\[ \xi-h_{J^{\sat}}((n-s)d-n-1),\]
i.e., the defect of $J^{\sat}$ in degree $(n-s)d-n-1$.
\end{proof}

The above method calculates (or bounds) the dimension of $H^n(X)_{\prim}$ for an $n-1$-dimensional hypersurface. The latter dimension equals the vanishing order of $1$ as a zero of its Alexander polynomial. To find the other vanishing orders one has to consider the $d$-fold cover of $\Ps^n$ ramified along $X$, which is a hypersurface in $\Ps^{n+1}$. (See also \cite[Remark 6.2.23]{Dim}.) At this stage we will again include the case $n=2$, i.e., assume now that $n\geq 2$. 

Let $\zeta_d=\exp(2\pi i/d)$.
We consider now the hypersurface $\tilde{X}\subset \Ps^{n+1}$ given by $y^d+f=0$. Let $\tilde{U}=\Ps^{n+1}\setminus \tilde{X}$. Let $T$ be the map  $y\mapsto \zeta_d^{-1}y$. Then  $T^*$ acts on $H^n(\tilde{U})$. Let $\tilde{S}=\tilde{X}_{\sing}$.
We have that  $q=(x_0:\dots:x_n:y) \in \tilde{S}$ if and only if $p=(x_0:\dots:x_n)\in S$ and $y=0$. 
In particular there is a natural bijection between $\tilde{S}$ and $S$.
The fix locus of the automorphism $T$ contains $\tilde{S}$. 
Moreover, for $q\in S$ then the induced linear map $T^*$ acts on $\C[x_0,\dots,x_n,y]$ and maps the Jacobian ideal of $y^d+f_p$ to itself, hence $T^*$ acts on $\cT_q$.

\begin{proposition}\label{prpOther} Suppose $k\in \{1,\dots,d-1\}$ and $s\in \{0,\dots, n-1\}$. Then the $\zeta_d^k$ eigenspace for $T^*$ acting on $\Gr_F^s H^{n+1}(\tilde{X})_{\prim}$ has dimension at most the defect of $J^{\sat}(f)$ in degree $(n+1-s)d-n-1-k$. The $1$-eigenspace of $T^*$ acting on $H^{n+1}(\tilde{X})_{\prim}$ is zero.
\end{proposition}

\begin{proof}
Recall that $y^{d-1}$ is in the Jacobian ideal both on the local and the global side, and all other partials do not involve $y$.
Let $q\in \tilde{S}$ and let $p$ be the corresponding point in $S$ then  $\cT_q=\oplus_{r=0}^{d-2} y^r \cT_p$ as $\C$-algebras. Recall that $T^*$ maps $\Omega_q$ to $\zeta^{-1}_d\Omega_q$. Obviously $T^*$ acts on $H^n(\tilde{U})$ and sends $\Omega$ to $\zeta_d^{-1}\Omega$,
Since $y^{d-1}\in J(y^d+f)$ we can decompose the Jacobian ring of $\tilde{X}$ as follows
\[ R[z]/J(y^d+f) \cong \oplus_{r=0}^{d-2} y^r R/J(f). \]

Consider now 
\[ H^{n+1}(\tilde{U})(1) \to  H^{n+1}_q(\tilde{X}).\]
As above we find a commutative diagram
\[\xymatrix{
 (R[z]/J(y^d+f))_{(n+1-s)d-n-2}  \ar[r]^{\;\;\;\;\;\;\tau}\ar[d]& \oplus_{q\in \tilde{S}} \cT_q \ar[d]\\
  \Gr_P^s H^{n+1}(\tilde{U})(1) \ar[r]^{\overline{\tau}}& \oplus_{q\in \tilde{S}}\Gr_P^s H^{n+1}_q(\tilde{X})}\]
Both vertical maps are surjective.
Each of the above maps is equivariant for $T^*$.

In particular the eigenvalues of $T^*$ on $H^{n+1}_q(\tilde{X})$ are all $d$-th roots of unity, but different from $1$ and therefore the $1$-eigenspace of $T^*$ acting on $H^{n+1}(\tilde{X})_{\prim}$ is zero.

From the above diagram it follows that  the cokernel of the bottom row is a quotient of the cokernel of the top row. Moreover the dimension of the  $\zeta^{-k}_d$-eigenspace is at most the cokernel of
\[ \tau_k:y^{k-1} R/J(f)_{(n+1-s)d-n-2-(k-1)} \to \oplus_{p\in S} y^{k-1}\cT_p.\]
 In a neighbourhood of $p$ one can identify $\Proj(R/J)$ with $\Spec(\cT_p)$. In particular the scheme  $V(J)$ is just $V(\oplus \cT_p)$. In particular, the kernel of the map of $\tau$ is just $(J^{\sat}/J)_{(n+1-s)d-n-2}$. Let $\xi$ be the length of $V(J)=V(J^{\sat})$, the total Tjurina number of the singularities. Then the cokernel of $\tau_k$ has dimension
\[ \xi-h_{J^{\sat}}((n-s)d-n-1-k)\]
which equals
the defect of $J^{\sat}$ in degree $(n+1-s)d-n-1-k$, by definition.
\end{proof}

\section{Calculation of Alexander polynomial}\label{SecComb}
In this section we will use the results of the previous section to calculate the Alexander polynomial to identify a range of degrees for which the ideal $J^{\sat}$ has defect.

\begin{definition}\label{defAlexPol}\cite[Definition 1.1.19]{Dim}
Let $n\geq 2$.
Let $f\in R$ be a homogeneous polynomial of degree $d$ such that $X=V(f)\subset \Ps^n$ has isolated singularities. Let $F=Z(f+1)\subset \C^{n+1}$ be the affine Milnor fiber of the singularity $(f,0)$. Then the \emph{Alexander polynomial} of $X$ is the characteristic polynomial of the monodromy operator acting on $H^{n-1}(F)$ and denoted by $\Delta_X(t)$.
\end{definition}

Consider $\tilde{X}=Z(y^d+f)\subset \Ps^{n+1}$.
Then the map $(x_0,\dots,x_n)\to (x_0:\dots:x_n:1)$ maps $F$ onto $\tilde{X}\setminus (\tilde{X}\cap Z(y))$. The set $\tilde{X}\cap Z(y)$ equals $X$. In this way we find an exact sequence
\[ 0 \to H^{n}_c(X)_{\prim} \to H^{n+1}_c(F)\to H^{n+1}_c(\tilde{X})_{\prim} \to 0.\]
The map $F \to \tilde{X}\setminus(\tilde{X}\cap Z(y))$ is an isomorphism. Using Poincar\'e duality on $F$ we obtain that
\[ H^{n-1}(F)\cong H^{n+1}_c(F)^ * \cong H^{n+1}_c(\tilde{X}\setminus (\tilde{X}\cap Z(y)))^*.\]
In particular, $H^{n-1}(F)$ is an extension of  $H^{n+1}_{\prim}(\tilde{X})^*$ by  $H^n(X)_{\prim}^*$.

As is shown in \cite[Remark 6.2.23]{Dim} we have that the monodromy operator on $H^{n+1}_c(F)$ is just the extension of the operator $T$ on $H^{n+1}(\tilde{X})_{\prim}$ and the identity map on $H^n(X)_{\prim}$.
Therefore the Alexander polynomial of $X$ is $(t-1)^a \varphi(t)$ with $a=h^n(X)_{\prim}$ and $\varphi(t)$ the characteristic polynomial of $T$ on $H^{n+1}(\tilde{X})_{\prim}$.
Since $T^d$ is the identity operator we have that all zeroes of the Alexander polynomial are $d$-th roots of unity, and from Proposition~\ref{prpOther} it follows that $\varphi(1)\neq 0$.

In the case that all singularities of $X$ are semi-weighted homogeneous, $H^n(X)$ has a pure weight $n$ Hodge structure and $H^{n+1}(\tilde{X})$ has a pure weight $n+1$ Hodge structure, see \cite{SteQua}.

Steenbrink \cite{SteenSpec} studied extensively the spectrum of polynomials with isolated singularities. The polynomial $f$ has a one-dimensional singular locus, hence there are two spectra, one associated with $H^{n}(F)$ and one with $H^{n-1}(F)$, see \cite[Section II.8.10]{Kuli}.

In the sequel we will call the spectrum associated with $H^{n-1}(F)$ the spectrum of $f$.
We will use the spectrum merely for bookkeeping reasons.
\begin{definition}
Let $\Z[\Q]$ be the group of formal sums of rational numbers, i.e., the set of expressions of the form $\sum_{\alpha\in \Q} n_\alpha [\alpha]$, with $n_{\alpha}\in \Z$ for all $\alpha$ and such that the set $\{\alpha\mid n_{\alpha}\neq 0\}$ is finite. The group law on $\Z[\Q]$ is the natural addition.

The \emph{spectrum} $sp(f)$ of $f$ is  the element $\sum n_\alpha [\alpha]$ of $ \Z[\Q]$ such that
\begin{enumerate}
\item If $\alpha\not \in [0,n]\cap \frac{1}{d}\Z$ then  $n_{\alpha}=0$,
\item If $\alpha$ is an integer then $n_\alpha=\dim \Gr_F^{n-\alpha} H^{n}(X)_{\prim}$, 
\item If $\alpha$ is not an integer, but $d\alpha$ is integer then let $s=\lceil \alpha \rceil$, and $k=d(s-\alpha)$. Then $n_\alpha$ equals the dimension of $\zeta_d^{-k}$ eigenspace for $T^*$ acting on $\Gr_F^{n+1-s} H^{n+1}(\tilde{X})$.
\end{enumerate}
\end{definition}

\begin{lemma}\label{lemSymDeg}
We have $n_{\alpha}=n_{n-\alpha}$ and  $\sum_{\alpha} n_{\alpha}=\deg(\Delta_X)$.
\end{lemma}
\begin{proof}
Suppose first that $\alpha$ is some integer. The Hodge structure on $H^n(X)$ is pure of weight $n$, see \cite{SteQua}. Hence $h^{\alpha ,n-\alpha}_{\prim}=h^{n-\alpha,\alpha}_{\prim}$. In particular,  we find that $n_\alpha=n_{n-\alpha}$.

Suppose now that $\alpha$ is not an integer. Let $s=\lceil \alpha \rceil$ and $k=d(s-\alpha)$. Let $\zeta_d=\exp(2\pi i/d)$.
The Hodge structure on $H^{n+1}(\tilde{X})$ is pure of weight $n+1$ by \cite{SteQua}. Complex conjugation maps the $\zeta_d^e$-eigenspace to the $\zeta_d^{d-e}$-eigenspace. In  particular, the $\zeta_d^k$-eigenspace on $H^{s,n+1-s}$ and the $\zeta_d^{d-k}$ eigenspace of $H^{n+1-s,s}$ have the same dimension, hence $n_{s-\frac{k}{d}}=n_{n+1-s-\frac{(d-k)}{d}}$. Since we have
\[n +1-s-\frac{(d-k)}{d}=n-\left(s-\frac{k}{d}\right),\]
the statement follows.

Since $(T^*)^d$ is the identity it follows that $T^*$ is diagonalizable on $H^{n+1}(\tilde{X})$ and all eigenvalues of $T^*$ are $d$-th roots of unity. Moreover $1$ is not an eigenvalue by Proposition~\ref{prpOther}.

Hence the sum of the dimensions of the eigenspaces of $H^n(X)_{\prim}$ and  of the eigenspaces of $H^{n+1}(\tilde{X})_{\prim}$ equal the total dimension which in turn equals the degree of the Alexander polynomial.
Hence
 \[\sum_\alpha n_{\alpha}=\deg(\Delta_X).\]
\end{proof}

Let $J$ be the Jacobian ideal of $f$. Then $J^{\sat}$ is the ideal of the scheme $V(J)$. Let $\xi$ be the length of this scheme.
Proposition~\ref{prpOne} and Proposition~\ref{prpOther} imply the following result:

\begin{proposition}\label{prpBoundn} Suppose $\alpha>1$. 
We have
\[ n_{\alpha}\leq \xi-h_{J^{\sat}}(\alpha d-n-1).\]
\end{proposition}
\begin{proof}
Suppose first that $\alpha$ is an integer. Then
\[ n_\alpha=\dim \Gr_F^{n-\alpha} H^{n}(X)_{\prim}\]
Proposition~\ref{prpOne} implies that the latter is at most the defect of $J^{\sat}$ in degree $(n-(n-\alpha))d-n-1$.
Suppose now that $\alpha$ is not an integer then write $s=\lceil \alpha\rceil$ and $k=d(s-\alpha)$.  Then $n_\alpha$ equals the dimension of $\zeta_d^{-k}$ eigenspace for $T^*$ acting on $\Gr_F^{n+1-s} H^{n+1}(\tilde{X})$. Proposition~\ref{prpOther} implies that this at most the defect of  $J^{\sat}$ in degree $(n+1-(n+1-s))d-n-1-k=(s-\frac{k}{d})d-n-1$.
\end{proof}

\begin{lemma} Let $\Sigma \subset \Ps^n$ be a zero-dimensional scheme of length $m$. Then
\[ \delta(k):=m-h_{I(\Sigma)}(k)\]
is decreasing as a function in $k$.
\end{lemma}
\begin{proof}
Choose coordinates on $\Ps^n$ such that $V(x_0)\cap \Sigma=\emptyset$.
The number $\delta(k)$ equals the dimension of the cokernel of the evaluation map
\[ev: R_k\to  \oplus_{p\in \Sigma} A_p\]
where $A_p$ is the affine coordinate ring of $\Sigma$ in an affine neighbourhood of $p$, i.e., obtained by setting $x_0=1$.

Let $f_1,\dots,f_m$ be a basis for the image in degree $k$ and let $F_1,\dots,F_k$ be elements such that $ev(F_i)=f_i$.

Then in degree $k+1$ we have that $ev(x_0F_i)=ev(F_i)=f_i$, hence the dimension of the image in degree $k+1$ is at least the dimension of the image in degree $k$.
\end{proof}

\begin{proposition}\label{propDefect}
Suppose $\alpha>1$ and $n_{\alpha}>0$. Then $J^{\sat}$ has defect in every degree $\leq \alpha d-n-1$.
\end{proposition}
\begin{proof}
If $n_{\alpha}>0$ then Proposition~\ref{prpBoundn} implies that  $J^{\sat}$ has defect in degree $ \alpha d-n-1$.

However, $J^{\sat}$ is the ideal of a zero dimensional projective scheme, and for such a scheme one has that the defect is a decreasing function in the degree by the previous lemma, hence  $J^{\sat}$ has defect in every degree  up to $ \alpha d-n-1$.
\end{proof}

In order to show that $J^{\sat}$ has defect in degree $d$ we need to find an $\alpha$ such that $n_{\alpha}>0$ and $\alpha d-n-1\geq d$. 
Using the symmetry of the spectrum we know that if for some $\alpha$ we have $n_{\alpha}>0$ then we can find an $\alpha\geq \frac{n}{2}$ with $n_{\alpha}>0$. However for $n=2$ and for $n\geq 3 $ and $d$ small this is insufficient to show that $J^{\sat}$ has defect in degree $d$. If we take into account the $k$ such that $\Delta_X$ has a primitive $k$-th root of unity as a zero then we find a slightly larger $\alpha$ contained in the interval $[\frac{n}{2},\frac{n+1}{2}]$.
To identify such an $\alpha$ we use the following notation:

\begin{definition} Let $k>2$ be an integer, such that $k\mid d$.
Let $\psi(k)$ be the largest integer $m$ such that $\gcd(m,k)=1$ and $m<\frac{k}{2}$. Define
\[
\alpha(n,d,k)=\left\{
\begin{array}{cl}
\frac{n}{2}  & \mbox {if } n \mbox{ is even and } k=1  \mbox{ or } n \mbox{ is odd and } k=2\\
\frac{n+1}{2} & \mbox {if } n\mbox{ is even and } k=2 \mbox{ or } n \mbox{ is odd and } k=1\\
\frac{n+1}{2}-\frac{1}{k} & \mbox{if } n \mbox{ is odd and } k>2.\\
\frac{n}{2}+\frac{\psi(k)}{k} & \mbox{if } n \mbox{ is even and } k>2.
\end{array}.\right.
\]
\end{definition}
Note that $\alpha(n,d,k)\geq \frac{n}{2}$.

For an integer $k$ let $\zeta_k:=\exp(2 \pi i/k)$.
\begin{proposition}\label{propBound} Suppose $\zeta_k$ is a root of the Alexander polynomial, then
 $n_{\alpha}$ is nonzero for some $\alpha$ at least $\alpha(n,d,k)$.
\end{proposition}

\begin{proof}
If $k=1$ then  by the symmetry property (Lemma~\ref{lemSymDeg}) $n_{\alpha}>0$ for some integer $\alpha\geq \frac{n}{2}$.

If $k=2$ then  by the symmetry property $n_{\alpha}>0$ for some $p+\frac{1}{2}$ with $p$ an integer $p\geq \frac{n-1}{2}$.

Suppose now $k>2$.
Since the Alexander polynomial is in $\Q[t]$ we have for each $i$ with $0<i<k$ and $\gcd(i,k)=1$ that the sum
\[ \sum_{j=0}^n n_{j+\frac{i}{k}}\]
is independent of $i$. Using the symmetry we find that there at least $\varphi(k)/2$ values of $\alpha$ occurring in the spectrum which are of the form $\frac{n}{2}+\frac{i}{k}$ with $\gcd(i,k)=1$ and $i>0$.

Recall that there are precisely $\varphi(k)/2$ such values of $\alpha$ in the interval $[n/2,(n+1)/2]$. The largest one equals $\frac{n}{2}+\frac{\psi(k)}{k}$ if $n$ is even and $\frac{n-1}{2}+\frac{k-1}{k}$ if $n$ is odd.
\end{proof}

We will now identify the values of $(n,d,k)$ such that $\alpha(n,d,k) d -n-1\geq d$.

\begin{lemma} \label{lemGeneral} Let $n\geq 3$. Suppose that one of the following conditions hold
\begin{enumerate}
\item $d\geq 8, n=3$; 
\item $d\geq 5, n=4$;
\item  $d\geq 4, n\geq 5$.
\end{enumerate}
Then $(\alpha(n,d,k)-1)d \geq n+1$.
\end{lemma}

\begin{proof}
By definition we have $\alpha(n,d,k)\geq \frac{n}{2}$. Hence we are fine if $d\geq \frac{2n+2}{n-2}=2+\frac{6}{n-2}$.
\end{proof}

\begin{lemma}  \label{lemSpecial}  Suppose $n=d=4$. If $k\neq 1 $ then $(\alpha(n,d,k)-1)d \geq n+1$.
\end{lemma}

\begin{proof}
Since $k$ divides $d$ and $k\neq 1$ we know that $k\in \{2,4\}$. The claim follows from $\alpha(4,4,2)=\frac{5}{2},\alpha(4,4,4)=\frac{9}{4}$.
\end{proof}

\begin{lemma} \label{lemSpecialA} Suppose $n=3$ and either
\begin{enumerate}
\item $k=1$ and $d\geq 4$ or
\item $k\geq 3$ and $k\neq d$ (then $d\geq 2k\geq 6$) or
\item $k=d$ and $d\geq 5$.
\end{enumerate}
Then $(\alpha(n,d,k)-1)d \geq n+1$.
\end{lemma}

\begin{proof}
Suppose that $k=1$. Recall that $\alpha(3,d,1)=2$. For $d\geq 4$ we have  $(\alpha(3,d,1)-1)d= d\geq 4=n+1$.

If $k\geq 3$ then $d=kj$ for some positive integer $j$. Recall that $\alpha(3,kj,j)=\frac{2k-1}{k}$. Hence $(\alpha-1) d= (k-1)j$. This is at least 4 if $j\geq 2$, or $j=1$ and $k\geq 5$. Hence we have to exclude the case $k=d$ and $k\in \{3,4\}$.
\end{proof}

\begin{remark}
Suppose now that $n=3$ and  $k=2$ and that we have $d=2j$ for some $j\geq 2$ (since we excluded $d=2$). Recall that $\alpha(3,2j,2)=\frac{3}{2}$. Hence $(\alpha(3,2j,2)-1)d=j$. In particular,  for $j=2,3$ (hence $d=4,6$) we have to exclude $k=2$. These are the only even values of $d$ between 3 and 7.
\end{remark}

\begin{lemma} \label{lemCubic} Suppose $d=3$ and $n\geq 3$. Moreover suppose that \[(n,k)\not \in \{(3,1),(4,1),(6,1),(3,3),(4,3),(5,3)\}.\]
Then $(\alpha(n,d,k)-1)d \geq n+1$.
\end{lemma} 

\begin{proof} Since $k$ is a divisor $d$ we know  $k\in \{1,3\}$.

Suppose first that $k=1$. 
If $n$ is odd then $(\alpha(n,d,k)-1)d=\frac{3(n-1)}{2}$. This is at least $n+1$ for $n\geq 5$.
If $n$ is even  then $(\alpha(n,d,k)-1)d=\frac{3n-6}{2}$. This is at least $n+1$ for $n\geq 7$.

Suppose now that $k=3$.
If $n$ is odd then $(\alpha(n,d,k)-1)d=\frac{3n-5}{2}$. This is at least $n+1$ for $n\geq 7$.
If $n$ is even  then $(\alpha(n,d,k)-1)d=\frac{3n-4}{2}$. This is at least $n+1$ for $n\geq 6$.
\end{proof}

%
%
%
%
%
\begin{lemma} \label{lemPlane} Suppose $n=2$, $d\geq 3$. Suppose that $k$ is not a pure prime power and that $(d,k)\not \in \{(6,6),(12,6)\}$.
Then $(\alpha(n,d,k)-1)d \geq n+1$.
\end{lemma}

\begin{proof}
Write $d=kj$. The smallest $k$ which is a not a pure prime power is $6$. In particular, $(\alpha(2,d,k)-1)d=j\psi(k)$. 

If $\psi(k)=1$ then $\varphi(k)=2$. The only $k\geq 6$ for which this is possible is $k=6$. If $k=6$ then $d\geq 18$ and therefore $j\geq 3$.
If $j=1$ and $\psi(k)=2$ then $\varphi(k)=4$ and $k$ is odd.  In particular, $k$ would be equal to $5$, which we excluded.
Hence one of $\psi(k)>2$, $j>2$ or $\psi(k)=j=2$ holds and $(\alpha(2,d,k)-1)d=j\psi(k)\geq 3$.
\end{proof}

\begin{proposition}
Let $(n,d,k)$ be integers such that $d\geq 3,n\geq 2$ and $k\geq 1$ is a divisor of $d$.
Suppose $X\subset \Ps^n$ is an irreducible hypersurface of degree $d$ with isolated semi-weighted homogeneous singularities such that $\zeta_k$ is a zero of the Alexander polynomial of $X$.

Moreover assume that we are not in one of the following cases
\begin{enumerate}
\item $n=2$, $d\in \{6,12\}$ and $k=6$;
\item $n\in \{3,4,6\}$, $d=3$ and $k=1$; 
\item $n\in \{3,4,5\}$, $d=3$ and $k=3$;
\item $n=3$, $d\in \{4,6\}$ and $k=2$;
\item $n=3$, $d=k=4$;
\item $n=4$, $d=4$ and $k=1$;
\end{enumerate}

Then $J^{\sat}$ has defect in degree $d$.
\end{proposition}

\begin{proof}
The main result of  \cite{ZarIrr} implies that if $n=2$ then $\Delta_X(\zeta_{p^r})\neq 0$ for any prime number $p$ and nonnegative integer $r$. Hence if $n=2$ then $k$ is not a prime power.

Since $\zeta_k$ is a zero of the Alexander polynomial we know by Propositions~\ref{propDefect} and~\ref{propBound} that $J^{\sat}$ has defect in any degree up to $\alpha(n,d,k) d-n-1$. From Lemmata~\ref{lemGeneral},  \ref{lemSpecial}, \ref{lemSpecialA}, \ref{lemCubic}, \ref{lemPlane} it follows that $\alpha(n,d,k) d-n-1\geq d$.
\end{proof}

\begin{theorem}\label{thmMain}
Let $(n,d,k)$ be integers such that $d\geq 3,n\geq 2$ and $k\geq 1$ is a divisor of $d$.
Suppose $X\subset \Ps^n$ is an irreducible hypersurface of degree $d$ with isolated semi-weighted homogeneous singularities such that $\zeta_k$ is a zero of the Alexander polynomial of $X$.

Moreover assume that we are not in one of the  cases (1)-(6) of the previous Proposition. Then the equianalytic deformation space of $X$ is not $T$-smooth.
\end{theorem}

\begin{proof}
From \cite[Section 1.1.4.1]{GLScurves} it follows that the equianalytic deformation space is $T$-smooth if and only if $J^{\sat}$ has no defect in degree $d$. \end{proof}

\section{Examples} \label{SecExa}

We start with a general construction.

\begin{example}\label{exaConstr}
Let $f\in \C[\gamma_1,\dots,\gamma_n]$ be a weighted homogeneous polynomial, smooth outside the origin, with rational weights $w_1,\dots,w_n$, such that $\deg(f)=1$. Assume that the Tjurina algebra is not trivial.
Let $v$ be the smallest positive integer such that $vw_i$ is an integer for all $i$.

Let $m\geq 1$ be an integer let $g_i$ be a general form of degree $mvw_i$. Then
\[ f(g_1,\dots,g_n)\]
is a homogeneous polynomial of degree $d=mv$.  Let $X=V(f(g_1,\dots,g_n))\subset \Ps^n$.

Assume now that the $g_i$ are chosen such that $g_1,\dots,g_n$ form a regular sequence. Then the singular locus of $X$ contains the complete intersection $S_0=V(g_1,\dots,g_n)$. Moreover, if the $g_i$ are sufficiently general then at each point of $S_0$ the local equation for the singular point for some choice of coordinates is $f=0$.
Note that $S_0$ consists of
\[ (mv)^n \prod w_i\]
points. We claim that the Alexander polynomial of $X$ is non-trivial.

In Proposition~\ref{prpOther} we showed that $n_{\alpha}$ is \emph{at most} the dimension of the cokernel of 
\[ R_{\alpha d-n-1}\to \oplus_{p\in S} \cT_p.\]
However, if all singularities are weighted homogeneous then $\cT_p$ is a graded algebra. We can use this to determine $n_{\alpha}$ precisely. I.e., 
Proposition~\ref{prpCoKern} together with Remark~\ref{rmkSplit} yield that  $n_{\alpha}$ \emph{equals} the dimension of the cokernel of
\begin{equation}\label{eqnCoKer} R_{\alpha d-n-1}\to \oplus_{p\in S} (\cT_p)_\alpha.\end{equation}
The choice of local coordinates to obtain the correct grading on $\cT_p$ is very tricky, basically because one has to pick a particular part of a Taylor expansion and this is very sensitive to coordinate changes.
However, this is not an issue for the smallest $\alpha$ occurring in the spectrum of the singularity $f$. For such an $\alpha$ we have that  $(\cT_p)_{\alpha}=\cT_p/m_p$, where $m_p$ is the maximal ideal of $p$. Changing coordinates would yield  an automorphism given by multiplication by a nonzero number.

The smallest number in the spectrum of the isolated singularity $f=0$ is the sum of the weights $\alpha=\sum w_i$.
We want to show that $n_\alpha>0$ for this $\alpha$.  
For each $p\in S_0:=Z(g_1,\dots g_n)$ we have  $(\cT_p)_{\alpha}=\cT_p/m_p$. 
 Hence the cokernel of (\ref{eqnCoKer}) equals the cokernel of the evaluation map
 \[ R_{\alpha d-n-1}\to \oplus_{p\in S_0 }\C. \]
Hence 
\[ n_\alpha = \left((mv)^n\prod w_i\right)-h_{I(S_0)}(\alpha d-n-1).\]
 Since $g_1,\dots,g_n$ define a scheme-theoretic complete intersection, the  ideal generated by them has the Koszul complex on $g_1,\dots, g_n$ as its resolution. The highest degree of any generator in the resolution is $\sum \deg(g_i)=d\sum w_i$. Hence the largest degree for which there is defect equals $(d\sum w_i)-n-1$, i.e., we know that there is defect in degree
\[ d\sum w_i-n-1=\alpha d -n-1.\]
Hence $n_\alpha>0$ and  $\exp(2\pi i \sum w_i)$ is a zero of the Alexander polynomial.
From $n_\alpha>0$ for $\alpha=\sum w_i$ it follows that $n_\alpha>0$ for $\alpha=n-\sum w_i$ by Lemma~\ref{lemSymDeg}. Hence $J^{\sat}$ has defect in degree $(n-\sum w_i)d-n-1=(n-\sum w_i)mv-n-1$.

If we additionally assume that all the 
weights $w_i$ are of the form $1/k_i$, with $k_i\in \Z$ then the Tjurina number of each singularity is $\prod (k_i-1)=\prod \frac{1-w_i}{w_i}$. Hence the total Tjurina number equals
\[ (mv)^n \prod_i (1-w_i)=d^n\prod_i(1-w_i) \]
This number is so large that it is not clear whether for fixed $(n,k_1,\dots,k_n,m)$ there exists a component of the space of degree $mv$-curves with $(mv)^n\prod w_i$ singularities analytically equivalent with $f=0$ and constant Alexander polynomial.

If we instead assume that $n=2$, $f=\gamma_1^2+\gamma_2^3$, $w_1=\frac{1}{2},w_2=\frac{1}{3}$, $\alpha=\frac{5}{6}$ then we recover the example of B. Segre of degree $6m$ curves with $6m^2$ cusps.
\end{example}

We will now give two examples to illustrate how the above construction is the limit of known constructions of hypersurfaces with deformation space whose dimension is larger than expected.
The first example below is due to Greuel, Lossen and Shustin~\cite[Proposition 3.4]{CasFunc}:
\begin{example}\label{exaCusps}
Fix $k\in \Z_{>0}$.
Let $d=6m$, pick non-negative integers  $a_1,b_1\leq 6m$, such that $a_1$ and $b_1$ are divisible  by $2$ and by  $3$ respectively. Let $a_2=3m-\frac{a_1}{2}$, $b_2=2m-\frac{b_1}{3}$.

Pick general homogeneous forms $f_1,f_2,g_1,g_2$ in $x_0,x_1,x_2$ of degree $a_1,a_2,b_1,b_2$. Consider the curve
\[ f_1f_2^2+g_1g_2^3=0\]
Then $J^{\sat}\supset (f_2,g_2^2)$. The latter ideal is a complete intersection ideal. This ideal has defect in degree $d$ if and only if
\[ a_2+2b_2\geq 6m+3,\]
which happens if and only if
\[\frac{a_1}{2}+\frac{2b_1}{3}\leq m-3.\]

If the forms are chosen sufficiently general then the  singular locus is $f_2=g_2=0$ and all singularities are ordinary cusps. The curve has non-constant Alexander polynomial if and only if $(f_2,g_2)$ has defect in degree $5m-3$, i.e., if $a_2+b_2\geq 5m$. However $a_2\leq 3m, b_2\leq 2m$, therefore we have $a_1=b_1=0$.

Hence for each choice of $a_1,b_1\geq 0$ satisfying 
\[\frac{a_1}{2}+\frac{2b_1}{3}\leq m-3\]
 we find examples with non-$T$-smooth deformation space, but only for $a_1=b_1=0$ we find examples with non-constant Alexander polynomial.
\end{example}

We can apply the idea behind Example~\ref{exaCusps} to Example~\ref{exaConstr}:
\begin{example}\label{exaDefo}
Again let $f\in \C[\gamma_1,\dots,\gamma_n]$ be a weighted homogeneous polynomial, smooth outside the origin, with weights $w_1,\dots,w_n$, such that $\deg(f)=1$. Assume that the Tjurina algebra is not trivial.

Write $f= \sum_{i=1}^s M_i(\gamma_1,\dots,\gamma_n)$, where each $M_i$ is some $\C^*$-multiple of a monomial in $\gamma_1,\dots,\gamma_n$. Consider  \[F(\beta_1,\dots,\beta_s;\gamma_1,\dots,\gamma_n):= \sum_{i=1}^s \beta_i M_i(\gamma_1,\dots,\gamma_n).\]
Let $v$ be the smallest positive integer such that $vw_i$ is an integer for all $i$.

Let $m_1\geq 1$ be an integer and let $g_i$ be a general form of degree $m_1vw_i$. Let $m_2$ be another integer and let $h_1,\dots, h_s$ be general forms of degree $m_2$.
Consider the hypersurface 
 \[ F(h_1,\dots,h_s;g_1,\dots,g_n):= \sum_{i=1}^s h_i M_i(g_1,\dots,g_n).\]
If the forms are sufficiently general then one has singularities with local equation $f=0$  along $g_1=\dots=g_n=0$ and no further singularities.
 In this case $J^{\sat}$ has defect in degree
\[ \left(n-\sum w_i\right)m_1v-n-1.\]
The  degree of the hypersurface equals $m_1v+m_2$.
Hence if
\[ m_2\leq (n-w-1)m_1v-n-1,\]
then the deformation space is not $T$-smooth. For $m_2=0$ we recover the example with a non-constant Alexander polynomial.

If we can separate variables in $f$, i.e., suppose we can write 
\[ f=\sum_{j=1}^s f_j(x_{i_j},x_{i_j+1},\dots,x_{i_j+k_j})\]
with $i_{j+1}=i_j+k_j+1$
 then we can apply this construction for each summand with different choices of $m_2$, under the condition that $m_1v+m_2$ is the same for each summand. 
 
In Example~\ref{exaCusps} we did this for  $f=\gamma_1^2+\gamma_2^3$ and we took $f_1=\gamma_1^2$ and $f_2=\gamma_2^3$.
\end{example}

In the case of plane curves with $A_2$-singularities, Segre considered a family with $6m^2$ cusps on a curve of degree $6m$. In this case the expected dimension of the deformation space equals
\[ \frac{(6m+1)(6m+2)}{2} -1-12m^2= 6m^2+9m,\]
which is definitely positive. However in the case of $m^n$ ordinary $r$-fold points on a degree $rm$ hypersurface in $\Ps^n$ we obtain that the expected dimension is
\[ d_{r,n}(m)=\binom{mr+n}{n} - m^n (r-1)^n-1\]
The leading coefficient  of $d_{r,n}(m)$ equals $r^n (\frac{1}{n!}-1)$. Hence for $m$ sufficiently large we have that the expected dimension is negative. Therefore  the mere existence is sufficient to prove that  the deformation space is not $T$-smooth. We will now give an example of hypersurfaces with ordinary $r$-fold points, for which  the Alexander polynomial is non-constant and the expected dimension is positive. 

\begin{example}
Fix integers $\ell,r$ both at least 2. Let $n=2\ell$. Fix another positive integer $m$.

Fix $t$ polynomials $g_1,\dots, g_{\ell}$ of degree $m$ such that $g_1,\dots,g_{\ell},x_{\ell+1},\dots, x_{2\ell}$ define a complete intersection, which, as a scheme, is reduced.
Then this complete intersection consists of $m^\ell$ points.

Pick now $t$ generic forms $h_1,\dots, h_\ell$ of degree $m(r-1)$ from the ideal \[(g_1,\dots, g_\ell,x_{\ell+1},\dots,x_{2\ell})^{r-1}.\]  
Consider now $X=V( \sum_{i=1}^\ell x_{t+i} h _i).$
Then at each point in \[Z(g_1,\dots,g_\ell,x_{\ell+1},\dots,x_{2\ell})\] we have an $r$-fold point, and if the $h_i$ are chosen sufficiently general then these points are ordinary $r$-fold points.

In this way we have $m^\ell$ points of order $r$. The Milnor number of an $r$-fold point is $(r-1)^{2\ell}$, hence the expected codimension equals $m^\ell(r-1)^{2\ell}$, whereas the space of polynomials of degree $m(r-1)+1$ has dimension 
$\binom{m(r-1)+1+2\ell}{2\ell}$. The former polynomial is a polynomial of degree $\ell$ in $m$, whereas the latter polynomial is a polynomial of degree $2\ell$ in $m$ with positive leading coefficient. Hence for $m$ sufficiently large the expected dimension
\[  \binom{m(r-1)+1+2\ell}{2\ell}-m^\ell (r-1)^{2\ell}-1 \]
is positive.

In this case we have that the $\ell$-plane $x_{\ell+1}=\dots=x_{2\ell}=0$ defines a nonzero class of Hodge type $(\ell,\ell)$ in $H^{2\ell}(X,\C)_{\prim}$. In particular $n_\ell\neq0$.
 For this reason we have that $J^{\sat}$ has defect in any degree $\leq \ell m(r-1)-\ell-1$. The latter quantity is at least $d=m(r-1)+1$.
\end{example}

\end{document}